\documentclass[11pt,letterpaper]{amsart}

\usepackage{amssymb}
\usepackage{enumerate}

\newtheorem{theorem}{Theorem}

\newtheorem{lemma}[theorem]{Lemma}

\newcommand{\F}{\mathbb{F}}

\DeclareMathOperator{\Tr}{Tr}
\DeclareMathOperator{\tr}{tr}
\DeclareMathOperator{\GL}{GL}
\DeclareMathOperator{\sign}{sign}
\DeclareMathOperator{\rank}{rank}
\DeclareMathOperator{\rad}{rad}

\newcommand{\abs}[1]{\lvert#1\rvert}

\begin{document}

\title{Hermitian rank distance codes}

\author{Kai-Uwe Schmidt}
\address{Department of Mathematics, Paderborn University, Warburger Str.\ 100, 33098 Paderborn, Germany}
\email{kus@math.upb.de}

\date{07 March 2017 (revised 17 August 2017)}


\begin{abstract}
Let $X=X(n,q)$ be the set of $n\times n$ Hermitian matrices over~$\F_{q^2}$. It is well known that $X$ gives rise to a metric translation association scheme whose classes are induced by the rank metric. We study $d$-codes in this scheme, namely subsets $Y$ of $X$ with the property that, for all distinct $A,B\in Y$, the rank of $A-B$ is at least $d$. We prove bounds on the size of a $d$-code and show that, under certain conditions, the inner distribution of a $d$-code is determined by its parameters. Except if $n$ and $d$ are both even and $4\le d\le n-2$, constructions of $d$-codes are given, which are optimal among the $d$-codes that are subgroups of $(X,+)$. This work complements results previously obtained for several other types of matrices over finite fields.
\end{abstract}

\maketitle

\section{Introduction}

Let $X$ be a set of matrices over a finite field with the same number of rows and columns. Given an integer~$d$, we consider subsets $Y$ of $X$ with the property that, for all distinct $A,B\in Y$, the rank of $A-B$ is at least~$d$. We call such a set a \emph{$d$-code} in $X$. For fixed~$d$, one is usually interested in $d$-codes containing as many elements as possible. Instances of this problem have been considered when $X$ is the set of unrestricted matrices~\cite{Del1978}, alternating matrices~\cite{DelGoe1975}, and symmetric matrices~\cite{Sch2010},~\cite{Sch2015}. In all these cases, association schemes have been used critically to establish combinatorial properties of $d$-codes. In particular, bounds on the size of $d$-codes were obtained, which are often attained by constructions. Such results have found several applications in other branches of coding theory.
\par
In this paper, we consider the case that $X=X(n,q)$ is the set of $n\times n$ Hermitian matrices over the finite field $\F_{q^2}$ with $q^2$ elements. Here, $q$ is a prime power and~$\F_{q^2}$ is equipped with the involution $x\mapsto x^q$. We use the association scheme of Hermitian matrices to prove that every $d$-code $Y$ that is an additive subgroup of $(X,+)$ satisfies
\begin{equation}
\abs{Y}\le q^{n(n-d+1)}.   \label{eqn:bound}
\end{equation}
In the case that $d$ is odd, we prove that the bound~\eqref{eqn:bound} also holds for $d$-codes that are not necessarily subgroups of $(X,+)$ and that, in case of equality in~\eqref{eqn:bound}, the inner distribution of $Y$ is uniquely determined. In the case that~$d$ is even, we show by example that the bound~\eqref{eqn:bound} can be surpassed by $d$-codes not having the subgroup property and prove a larger bound that also holds for $d$-codes that are not necessarily subgroups of $(X,+)$. We also provide constructions of $d$-codes that are subgroups of $(X,+)$ and satisfy the bound~\eqref{eqn:bound} with equality for all possible $n$ and $d$, except if $n$ and $d$ are both even and $3<d<n$.
\par
It should be noted that related, but different, rank properties of sets of Hermitian matrices have been studied in~\cite{DumGowShe2011} and~\cite{GowLavSheVan2014}.


\section{The association scheme of Hermitian matrices}

A (symmetric) \emph{association scheme} with $n$ classes is a finite set $X$ together with $n+1$ nonempty relations $R_0,R_1,\dots,R_n$ that partition $X\times X$ and satisfy:
\begin{enumerate}
\setlength{\itemsep}{.5ex}
\item[(A1)] $R_0$ is the identity relation;
\item[(A2)] each of the relations is symmetric;
\item[(A3)] if $(x,y)\in R_k$, then the number of $z\in X$ such that $(x,z)\in R_i$ and $(z,y)\in R_j$ is a constant $p^k_{ij}$ depending only on $i$, $j$, and $k$, but not on the particular choice of $x$ and~$y$.
\end{enumerate}
For background on association schemes and connections to coding theory we refer to~\cite{Del1973},~\cite{DelLev1998}, and~\cite{MarTan2009} and to~\cite[Chapter 21]{MacSlo1977} and~\cite[Chapter 30]{vLiWil2001} for gentle introductions.
\par
Let $q$ be a prime power and let $\overline{x}=x^q$ be the conjugate of $x\in\F_{q^2}$. For a matrix $A$ over $\F_{q^2}$, write $A^*$ for the matrix obtained from $A$ by conjugation of each entry and transposition. An $n\times n$ matrix $A$ with entries in $\F_{q^2}$ is \emph{Hermitian} if $A^*=A$. Let $X=X(n,q)$ denote the set of $n\times n$ Hermitian matrices over $\F_{q^2}$. Then $X$ is an $n^2$-dimensional vector space over $\F_q$.
\par
It is well known~\cite[Section~9.5]{BroCohNeu1989} that $X$ gives rise to an association scheme with $n$ classes whose relations are given by
\[
(A,B)\in R_i \Leftrightarrow \rank(A-B)=i.
\]
Alternatively these relations arise as orbits of a group action. Let $G=\GL_n(\F_{q^2})\rtimes X$ be the semidirect product of the general linear group $\GL_n(\F_{q^2})$ and $X$, so that $G$ acts transitively on $X$ as follows
\[
\begin{split}
G\times X&\to X\\
((T,D),A)&\mapsto TAT^*+D.
\end{split}
\]
The action of $G$ extends to $X\times X$ componentwise and so partitions $X\times X$ into orbits, which are the relations defined above (see~\cite[Chapter~6]{Wan1996}, for example).
\par
The relations just defined are invariant under the translation $(A,B)\mapsto(A+C,B+C)$, which is the defining property of a \emph{translation scheme}. We shall make heavy use of the eigenvalues of this translation scheme, which are determined by the characters of $(X,+)$~\cite[Section~V]{DelLev1998}. Let $\chi:\F_q\to\mathbb{C}$ be a nontrivial character of $(\F_q,+)$ and, for $A,B\in X$, write
\[
\langle A,B\rangle=\chi(\tr(A^*B)),
\]
where $\tr$ is the matrix trace. For all $A,A',B\in X$, we have
\begin{equation}
\langle A+A',B\rangle=\langle A,B\rangle\langle A',B\rangle.   \label{eqn:ip_homomorphism}
\end{equation}
Indeed, it is readily verified that the mapping $A\mapsto\langle A,B\rangle$ ranges through all characters of $(X,+)$ as $B$ ranges over $X$. Let $X_i$ be the subset of $X$ containing all matrices of rank $i$. For $i,k\in\{0,1,\dots,n\}$, the numbers
\begin{equation}
Q_k(i)=\sum_{A\in X_k}\langle A,B\rangle\quad\text{for $B\in X_i$}   \label{eqn:eigenvalues}
\end{equation}
are independent of the choice of $B$ and are the \emph{eigenvalues} of the association scheme defined above (see~\cite[Section~V]{DelLev1998} for details). For odd $q$, these numbers have been determined by Carlitz and Hodges~\cite{CarHod1955} and also by Stanton~\cite{Sta1981}. We shall require the eigenvalues in the following form
\begin{equation}
Q_k(i)=(-1)^k\sum_{j=0}^k{n-j\brack n-k}{n-i\brack j}(-q)^{{k-j\choose 2}+nj},   \label{eqn:explicit_ev}
\end{equation}
where, for integral $m$ and $\ell$ with $\ell\ge 0$,
\[
{m\brack \ell}=\prod_{i=1}^\ell((-q)^{m-i+1}-1)/((-q)^i-1).
\]
is the \emph{negative $q$-binomial coefficient}. A simple proof of the formula~\eqref{eqn:explicit_ev} for odd and even $q$ is given in the appendix.
\par
Equivalently, the eigenvalues are given by the $n+1$ equations
\begin{equation}
\sum_{k=0}^j{n-k\brack n-j}Q_k(i)=(-1)^{(n+1)j}q^{nj}{n-i\brack j}   \label{eqn:Q_inversion}
\end{equation}
for $j\in\{0,1,\dots,n\}$, which can be proved using the inversion formula
\begin{equation}
\sum_{j=i}^k(-1)^{j-i}(-q)^{j-i\choose 2}{j\brack i}{k\brack j}=\delta_{k,i}   \label{eqn:inversion}
\end{equation}
(see~\cite[(10)]{DelGoe1975}, for example), where $\delta_{k,i}$ is the Kronecker $\delta$-function.


\section{Combinatorial properties of subsets of $X(n,q)$}

Let $Y$ be a nonempty subset of $X=X(n,q)$. The \emph{inner distribution} of~$Y$ is the tuple $(A_0,A_1,\dots,A_n)$ of rational numbers, which are given by
\[
A_i=\frac{\abs{(Y\times Y)\cap R_i}}{\abs{Y}}.
\]
In other words, $A_i$ is the average number of pairs in $Y\times Y$ whose difference has rank $i$. Note that we always have $A_0=1$. The \emph{dual inner distribution} of $Y$ is the tuple $(A'_0,A'_1,\dots,A'_n)$, whose entries are given by
\begin{equation}
A'_k=\sum_{i=0}^nQ_k(i)A_i.   \label{eqn:def_dual_dist}
\end{equation}
Then $A'_0=\abs{Y}$ and, as a consequence of a general property of association schemes (see~\cite[Theorem~3]{DelLev1998}, for example), we have
\begin{equation}
A'_k\ge 0\quad\text{for each $k\in\{0,1,\dots,n\}$}.   \label{eqn:Ak_nonnegative}
\end{equation}
Given an integer $d$ satisfying $1\le d\le n$, we say that $Y$ is a \emph{$d$-code} if $A_1=\dots=A_{d-1}=0$. Equivalently, $Y$ is a $d$-code if $\rank(A-B)\ge d$ for all distinct $A,B\in Y$. We say that $Y$ is a \emph{$t$-design} if $A'_1=\dots=A'_t=0$.
\par
Now suppose that $Y$ is a subgroup of $(X,+)$. In this case, we say that $Y$ is \emph{additive}. It is readily verified that, if $Y$ has inner distribution $(A_0,A_1,\dots,A_n)$, then $A_i$ counts the number of matrices in $Y$ of rank $i$. We can associate with $Y$ its \emph{dual}
\[
Y^\perp=\{B\in X:\langle A,B\rangle=1\;\text{for each $A\in Y$}\},
\]
which is also additive and satisfies
\[
\abs{Y}\,\abs{Y^\perp}=\abs{X}.
\]
It follows from a well known property of association schemes (see~\cite[Theorem~27]{DelLev1998}, for example) that, if $Y$ has dual inner distribution $(A'_0,A'_1,\dots,A'_n)$, then the tuple
\[
\frac{1}{\abs{Y}}(A'_0,A'_1,\dots,A'_n)
\]
is the inner distribution of $Y^\perp$. This implies in particular that the entries in the dual inner distribution of an additive set $Y$ are divisible by $\abs{Y}$.
\par
We use this fact and the property~\eqref{eqn:Ak_nonnegative} to prove bounds on the size of $d$-codes.
\begin{theorem}
\label{thm:bound}
Every additive $d$-code $Y$ in $X(n,q)$ satisfies
\[
\abs{Y}\le q^{n(n-d+1)}.
\]
Moreover, if $d$ is odd, then this bound also holds for arbitrary $d$-codes $Y$ in $X(n,q)$ and equality holds if and only if $Y$ is an $(n-d+1)$-design.
\end{theorem}
\begin{proof}
Let $(A_0,\dots,A_n)$ and $(A'_0,\dots,A'_n)$ be the inner distribution and the dual inner distribution of $Y$, respectively. Use~\eqref{eqn:def_dual_dist} and~\eqref{eqn:Q_inversion} to obtain, for each $j\in\{0,1,\dots,n\}$,
\begin{align*}
\sum_{k=0}^j{n-k\brack n-j}A'_k&=\sum_{i=0}^nA_i\sum_{k=0}^j{n-k\brack n-j}Q_k(i)\\
&=(-1)^{(n+1)j}q^{nj}\sum_{i=0}^nA_i{n-i\brack j}.
\end{align*}
Set $j=n-d+1$ and use $A_0=1$ and $A_1=\dots=A_{d-1}=0$ and the fact that ${m\brack \ell}=0$ for $m<\ell$ to find that
\begin{equation}
\sum_{k=0}^{n-d+1}{n-k\brack d-1}A'_k=(-1)^{(n+1)(n-d+1)}q^{n(n-d+1)}{n\brack d-1}.   \label{eqn:sum_A_k}
\end{equation}
If $Y$ is additive, then the left-hand side is divisible by $\abs{Y}$, hence the right-hand side is divisible by $Y$. Let $p$ be the prime dividing $q$. If $Y$ is additive, then $\abs{Y}$ is a power of $p$. It is readily verified that ${n\brack d-1}$ is not divisible by~$p$, which implies that~$\abs{Y}$ divides $q^{n(n-d+1)}$, proving the bound for additive codes.
\par
Now let $d$ be odd. Note that the sign of ${m\brack \ell}$ equals $(-1)^{\ell(m-\ell)}$. Hence, since $d$ is odd, the binomial coefficients in the sum on the left-hand side of~\eqref{eqn:sum_A_k} are nonnegative. Since the numbers $A'_k$ are also nonnegative by~\eqref{eqn:Ak_nonnegative} and $A'_0=\abs{Y}$, we find from~\eqref{eqn:sum_A_k} that
\[
{n\brack d-1}\abs{Y}\le(-1)^{(n+1)(n-d+1)}q^{n(n-d+1)}{n\brack d-1},
\]
which gives the bound for general $d$-codes in the case that $d$ is odd. Finally, equality occurs if and only if $A'_1=\dots=A'_{n-d+1}=0$ in~\eqref{eqn:sum_A_k}, which is equivalent to $Y$ being an $(n-d+1)$-design.
\end{proof}
\par
For even $d$, the bound given in Theorem~\ref{thm:bound} cannot hold in general for arbitrary $d$-codes in $X(n,q)$. For example, Theorem~\ref{thm:bound} asserts that the largest additive $n$-code in $X(n,q)$ has size $q^n$, whereas there exist $n$-codes in $X(n,q)$ of size $q^n+1$. This will be shown in Theorem~\ref{thm:con_non_additive}.
\par
The best bound we could prove for $d$-codes when $d$ is even is contained in the following theorem.
\begin{theorem}
\label{thm:lp_deven}
For even $d$, every $d$-code $Y$ in $X(n,q)$ satisfies
\[
\abs{Y}\le (-1)^{n+1}q^{n(n-d+1)}\frac{((-q)^{n-d+2}-1)+(-q)^n((-q)^{n-d+1}-1)}{(-q)^{n-d+2}-(-q)^{n-d+1}}.
\]
\end{theorem}
\begin{proof}
Let $(A_0,\dots,A_n)$ and $(A'_0,\dots,A'_n)$ be the inner distribution and the dual inner distribution of $Y$, respectively. As in the proof of Theorem~\ref{thm:bound}, we have for each $j\in\{0,1,\dots,n\}$,
\[
\sum_{k=0}^j{n-k\brack n-j}A'_k=(-1)^{(n+1)j}q^{nj}\sum_{i=0}^nA_i{n-i\brack j}.
\]
Apply this identity with $j=n-d+1$ and $j=n-d+2$ to obtain, as in the proof of Theorem~\ref{thm:bound},
\begin{align*}
\sum_{k=0}^{n-d+1}{n-k\brack d-1}A'_k&=(-1)^{(n+1)(n-d+1)}q^{n(n-d+1)}{n\brack d-1}
\intertext{and}
\sum_{k=0}^{n-d+2}{n-k\brack d-2}A'_k&=(-1)^{(n+1)(n-d+2)}q^{n(n-d+2)}{n\brack d-2}.
\end{align*}
Notice that  we can extend the summation range in the first identity up to $n-d+2$ without changing the value of the sum. Therefore, writing
\[
u_k={n-k\brack d-1}{n-1\brack d-2}\quad\text{and}\quad v_k={n-k\brack d-2}{n-1\brack d-1}
\]
and using that $d$ is even, we find that
\begin{multline}
(-1)^{n+1}\sum_{k=0}^{n-d+2}(u_k-v_k)A'_k\\
=q^{n(n-d+1)}\Bigg({n\brack d-1}{n-1\brack d-2}+(-q)^n{n\brack d-2}{n-1\brack d-1}\Bigg).   \label{eqn:sum_Ak}
\end{multline}
Next we show that the summands on the left-hand side are nonnegative. Since the sign of ${m\brack \ell}$ is $(-1)^{\ell(m-\ell)}$, we find that $\sign(u_k)=(-1)^{n-k+1}$ and $\sign(v_k)=(-1)^n$. Therefore the left-hand side of~\eqref{eqn:sum_Ak} equals
\[
\sum_{k=0}^{n-d+2}\big((-1)^k\abs{u_k}+\abs{v_k}\big)A'_k.
\]
We have
\[
\frac{u_k}{v_k}=\frac{(-q)^{n-k-d+2}-1}{(-q)^{n-d+1}-1},
\]
from which we find that $\abs{u_k}\le \abs{v_k}$ for each $k\ge 1$. Hence the left-hand side of~\eqref{eqn:sum_Ak} can be bounded from below by
\[
(-1)^{n+1}\bigg({n\brack d-1}{n-1\brack d-2}-{n\brack d-2}{n-1\brack d-1}\bigg)A'_0.
\]
Since this expression is positive and $A'_0=\abs{Y}$, we obtain
\[
\abs{Y}\le (-1)^{n+1}q^{n(n-d+1)}\frac{\displaystyle{n\brack d-1}{n-1\brack d-2}+(-q)^n{n\brack d-2}{n-1\brack d-1}}{\displaystyle{n\brack d-1}{n-1\brack d-2}-{n\brack d-2}{n-1\brack d-1}},
\]
from which the desired bound can be obtained after elementary manipulations.
\end{proof}
\par
For example, for $d=n=2$, the bound of Theorem~\ref{thm:lp_deven} is
\[
\abs{Y}\le q^3-q^2+q.
\]
It is known that this bound is not tight; the largest $2$-code in $X(2,q)$ has size $5$, $16$, $24$, $47$ for $q$ equal to $2$, $3$, $4$, $5$, respectively~\cite{Sch2016}. In these cases, the optimal codes have been classified in~\cite{Sch2016}. For $q=2$, the unique optimal construction arises as a special case of Theorem~\ref{thm:con_non_additive}.
\par
However, it is conjectured that Theorem~\ref{thm:lp_deven} gives the optimal solution to the linear program, whose objective is to maximise 
\[
\abs{Y}=\sum_{i=0}^nA_i,
\]
subject to the nonnegativity of the numbers $A_i$ and $A'_k$ attached to $Y$. This has been checked with a computer for many small values of $n$ and $q$.
\par
It is well known~\cite[Lemma~1]{GowLavSheVan2014} that there exists an $n$-code in $X(n,q)$ of size $N$ if and only if there exists a partial spread in the Hermitian polar space $H(2n-1,q^2)$ of size $N+1$. We can therefore obtain bounds for $n$-codes in $X(n,q)$ from bounds for partial spreads in $H(2n-1,q^2)$ and vice versa. For example, Theorem~\ref{thm:bound} implies that, for odd $n$, a partial spread in $H(2n-1,q^2)$ contains at most $q^n+1$ elements. This gives another proof of a theorem due to Vanhove~\cite{Van2009},~\cite{Van2011}. In the other direction, from a result due to De~Beule, Klein, Metsch, and Storme~\cite{DebKleMetSto2008} we obtain
\[
\abs{Y}\le\frac{q(q^2+1)}{2}
\]
for every $2$-code $Y$ in $X(2,q)$. This bound is tight for $q\in\{2,3\}$. From a result due to Ihringer~\cite{Ihr2014} we have
\[
\abs{Y}\le \frac{q^{2n}-1}{q+1}
\]
for every $n$-code in $X(n,q)$, which can be proved more directly using~\cite[Corollary~3.2]{Ihr2014} together with the explicit knowledge of the eigenvalues~\eqref{eqn:explicit_ev}, in particular~\eqref{eqn:Qk0} and~\eqref{eqn:num_her_mat} for $k=1$. This bound is slightly better than the corresponding bound $\abs{Y}\le q^{2n-1}-q^n+q^{n-1}$ of Theorem~\ref{thm:lp_deven}. Some improved bounds for $n$-codes in $X(n,q)$ in the case that $q$ is not a prime can be obtained from~\cite{IhrSinXia2018}.
\par
Our final result of this section gives the inner distribution of a $d$-code, provided that it is also an $(n-d)$-design.
\begin{theorem}
\label{thm:inner_dist}
If $Y$ is a $d$-code and an $(n-d)$-design in $X(n,q)$, then its inner distribution $(A_i)$ satisfies 
\[
A_{n-i}=\sum_{j=i}^{n-d}(-1)^{j-i}(-q)^{j-i\choose 2}{j\brack i}{n\brack j}\bigg(\frac{\abs{Y}}{q^{nj}}(-1)^{(n+1)j}-1\bigg)
\]
for each $i\in\{0,1,\dots,n-1\}$.
\end{theorem}
\begin{proof}
Let $(A_0,\dots,A_n)$ and $(A'_0,\dots,A'_n)$ be the inner distribution and the dual inner distribution of $Y$, respectively. As in the proof of Theorem~\ref{thm:bound}, we have for each $j\in\{0,1,\dots,n\}$,
\[
\sum_{k=0}^j{n-k\brack n-j}A'_k=(-1)^{(n+1)j}q^{nj}\sum_{i=0}^nA_i{n-i\brack j}.
\]
Since $Y$ is a $d$-code and an $(n-d)$-design, we find that, for each $j\in\{0,1,\dots,n-d\}$,
\[
{n\brack j}\bigg(\frac{\abs{Y}}{q^{nj}}(-1)^{(n+1)j}-1\bigg)=\sum_{i=0}^{n-d}A_{n-i}{i\brack j}.
\]
The proof is completed by applying the inversion formula~\eqref{eqn:inversion}.
\end{proof}
\par
Call a $d$-code $Y$ in $X(n,q)$ \emph{maximal additive} if $Y$ is additive and
\[
\abs{Y}=q^{n(n-d+1)},
\]
so that $Y$ meets the bound of Theorem~\ref{thm:bound} with equality. If $d$ is odd, then $Y$ is an $(n-d+1)$-design by Theorem~\ref{thm:bound}, and so Theorem~\ref{thm:inner_dist} implies that the inner distribution of $Y$ is uniquely determined by its parameters. The situation is different for even $d$. It was checked with a computer that there are exactly four different inner distributions of maximal additive $2$-codes in $X(3,2)$ and at least three different inner distributions of maximal additive $2$-codes in $X(4,2)$. The four possibilities for the inner distribution $(A_0,A_1,A_2,A_3)$ of a maximal additive $2$-code in $X(3,2)$ are
\[
(1,0,21,42),\quad (1,0,29,34),\quad (1,0,37,26),\quad (1,0,45,18).
\]


\section{Constructions}

Recall from the previous section that a maximal additive $d$-code in $X(n,q)$ is a $d$-code that meets the bound of Theorem~\ref{thm:bound} with equality. In this section, we provide constructions of maximal additive $d$-codes in $X(n,q)$ for all possible values of $d$, except when $n$ and $d$ are both even and $4\le d\le n-2$.
\par
We shall work with Hermitian forms rather than with matrices. Let $V=V(n,q^2)$ be an $n$-dimensional vector space over $\F_{q^2}$. Recall that a \emph{Hermitian form} on $V$ is a mapping
\[
H:V\times V\to \F_{q^2}
\]
that is $\F_{q^2}$-linear in the first coordinate and satisfies $H(y,x)=\overline{H(x,y)}$ for all $x,y\in V$. The \emph{(left) radical} of a Hermitian form $H$ on $V$ is the $\F_{q^2}$-vector space
\[
\rad(H)=\{x\in V:H(x,y)=0\;\text{for all $y\in V$}\}
\]
and its \emph{rank} is $n-\dim\rad(H)$. Fixing a basis $\xi_1,\dots,\xi_n$ for $V$ over $\F_{q^2}$, we can identify a Hermitian form $H$ on $V$ with the $n\times n$ Hermitian matrix 
\[
(H_{ij}=H(\xi_i,\xi_j))_{1\le i,j\le n}.
\]
It is readily verified that the rank of this matrix equals the rank of the Hermitian form $H$. In fact this gives a one-to-one correspondence between $X(n,q)$ and Hermitian forms on $V$.
\par
We shall identify the vector space $V(n,q^2)$ with $\F_{q^{2n}}$ and use the relative trace function $\Tr:\F_{q^{2n}}\to\F_{q^2}$, given by
\[
\Tr(x)=\sum_{k=0}^{n-1}x^{q^{2k}}.
\]
It is easy to check that this trace function is $\F_{q^2}$-linear and satisfies $\Tr(x)^q=\Tr(x^q)$ for all $x\in\F_{q^{2n}}$.
\par
The following theorem contains a construction for maximal additive $d$-codes in $X(n,q)$ when $n-d$ is odd.
\begin{theorem}
\label{thm:construction_1}
Let $n$ and $d$ be integers of opposite parity satisfying $1\le d\le n-1$. Then, as $a_1,\dots,a_{(n-d+1)/2}$ range over $\F_{q^{2n}}$, the mappings
\begin{gather*}
H:\F_{q^{2n}}\times \F_{q^{2n}}\to\F_{q^2}\\[1ex]
H(x,y)=\Tr\bigg(\sum_{j=1}^{(n-d+1)/2}(a_jxy^{q^{2j-1}}+(a_j)^qx^{q^{2j}}y^q)\bigg)
\end{gather*}
form an additive $d$-code in $X(n,q)$ of size $q^{n(n-d+1)}$.
\end{theorem}
\begin{proof}
It is readily verified that the mappings $H$ are Hermitian and that the linearity of the trace function implies that the set under consideration is additive. It is therefore enough to show that $H$ has rank at least $d$ unless $a_1=\dots=a_{(n-d+1)/2}=0$.
\par
We may write
\[
H(x,y)=\Tr(y^q\,L(x)),
\]
where $L$ is an endomorphism of $\F_{q^{2n}}$, given by
\[
L(x)=\sum_{j=1}^{(n-d+1)/2}\big((a_jx)^{q^{2n-2j+2}}+(a_j)^q\,x^{q^{2j}}\big).
\]
We have
\[
L(x^{q^{n-d-1}})=\sum_{j=1}^{(n-d+1)/2}\big((a_j)^{q^{2n-2j+2}}x^{q^{n-d-2j+1}}+(a_j)^q\,x^{q^{n-d+2j-1}}\big).
\]
If not all of $a_j$'s are zero, then this is a polynomial of degree at most $q^{2(n-d)}$ and so has at most $q^{2(n-d)}$ zeros. Now notice that
\begin{align*}
\F_{q^{2n}}\times \F_{q^{2n}}&\to \F_{q^2}\\
(u,v)&\mapsto \Tr(uv)
\end{align*}
is a nondegenerate bilinear form. Therefore, since the kernel of a nonzero $L$ on $\F_{q^{2n}}$ has dimension at most $n-d$ over $\F_{q^2}$, the radical of the corresponding Hermitian form also has dimension at most $n-d$ over $\F_{q^2}$. Therefore, $H$ has rank at least $d$ unless $a_1=\dots=a_{(n-d+1)/2}=0$, as required.
\end{proof}
\par
The following theorem contains a construction for $d$-codes in $X(n,q)$ when $n$ and $d$ are both odd.
\begin{theorem}
\label{thm:construction_2}
Let $n$ and $d$ be odd integers satisfying $1\le d\le n$. Then, as $a_0$ ranges over $\F_{q^n}$ and $a_1,\dots,a_{(n-d)/2}$ range over $\F_{q^{2n}}$, the mappings
\begin{gather*}
H:\F_{q^{2n}}\times \F_{q^{2n}}\to\F_{q^2}\\[1ex]
H(x,y)=\Tr\bigg(a_0xy^{q^n}+\sum_{j=1}^{(n-d)/2}(a_jxy^{q^{n-2j}}+(a_j)^qx^{q^{n-2j+1}}y^q)\bigg)
\end{gather*}
form an additive $d$-code in $X(n,q)$ of size $q^{n(n-d+1)}$.
\end{theorem}
\begin{proof}
The proof is similar to that of Theorem~\ref{thm:construction_1}, and so only a sketch is included. We may write
\[
H(x,y)=\Tr(y^q\,L(x)),
\]
where $L$ is an endomorphism of $\F_{q^{2n}}$, given by
\[
L(x)=(a_0x)^{q^{n+1}}+\sum_{j=1}^{(n-d)/2}\big((a_jx)^{q^{n+2j+1}}+(a_j)^q\,x^{q^{n-2j+1}}\big).
\]
If not all of the $a_j$'s are zero, then $L(x^{q^{2n-d-1}})$ is induced by a polynomial of degree at most $q^{2(n-d)}$ and therefore, as in the proof of Theorem~\ref{thm:construction_1}, we find that $H$ hast rank at least $d$ unless $a_0=\dots=a_{(n-d)/2}=0$.
\end{proof}
\par
Theorems~\ref{thm:construction_1} and~\ref{thm:construction_2} give constructions of maximal additive $d$-codes in $X(n,q)$ for every possible $n$ and $d$ except when both $n$ and $d$ are even. Constructions of maximal additive $d$-codes in $X(n,q)$ are easy to obtain for $d=2$ and for $d=n$, independently of whether $n$ is even or odd. For $d=n$, we can take an $\F_q$-vector space of $q^n$ symmetric matrices of size~$n\times n$ over~$\F_q$ with the property that every nonzero matrix in this space is nonsingular. Constructions of such sets are well known (see~\cite{GabPil2006} or~\cite{Sch2015}, for example). Another construction of maximal additive $n$-codes in $X(n,q)$ was given in~\cite{DumGowShe2011}. For $d=2$, we can take all matrices in $X(n,q)$ whose main diagonal contains only zeros~\cite[Theorem~6.1]{Sch2016}. However, it is currently an open problem how to construct (if they exist) maximal additive $d$-codes in $X(n,q)$ when~$n$ and~$d$ are even integers satisfying $4\le d\le n-2$.
\par
We close this section by showing that the bound for additive codes in Theorem~\ref{thm:bound} can be surpassed by non-additive codes whenever $n$ is even and $d=n$. This follows already from~\cite[Theorem~9]{GowLavSheVan2014}. Here we give a more direct construction. The main ingredient is a set $Z$ of $m\times m$ matrices over~$\F_q$ with the property that $\abs{Z}=q^m$ and $A-B$ is nonsingular for all distinct $A,B\in Z$. Such objects are equivalent to finite quasifields~\cite{delKieWasWil2016} and several constructions are known (see~\cite{Del1978} for a canonical construction corresponding to finite fields).
\begin{theorem}
\label{thm:con_non_additive}
Let $n$ be an even positive integer and let $Z$ be a set of $q^n$ matrices over $\F_{q^2}$ of size $n/2\times n/2$ with the property that $A-B$ is nonsingular for all distinct $A,B\in Z$. Let
\[
Y=\left\{
\begin{pmatrix}
I & A^*\\
A & AA^*
\end{pmatrix}
:A\in Z
\right\}
\cup
\left\{
\begin{pmatrix}
O & O\\
O & I
\end{pmatrix}
\right\},
\]
where $O$ and $I$ are the zero and identity matrices of size $n/2\times n/2$, respectively. Then $Y$ is an $n$-code in $X(n,q)$ of size $q^n+1$.
\end{theorem}
\begin{proof}
By the assumed properties of $Z$, it is plain that
\[
\begin{pmatrix}
O & A^*-B^*\\
A-B & AA^*-BB^*
\end{pmatrix}
\]
is nonsingular for all distinct $A,B\in Z$. Moreover, for each $n/2\times n/2$ matrix~$A$ over $\F_{q^2}$, we have
\[
\begin{pmatrix}
I &  O\\
-A & I
\end{pmatrix}
\begin{pmatrix}
I & A^*\\
A & AA^*-I
\end{pmatrix}
=
\begin{pmatrix}
I & A^*\\
O & -I
\end{pmatrix},
\]
and the proof is completed.
\end{proof}


\appendix

\section{Computation of the eigenvalues}

We now derive the explicit expressions~\eqref{eqn:explicit_ev} for the numbers $Q_k(i)$. We begin with the following lemma, which gives a recurrence formula for the eigenvalues. Write $Q^{(n)}_k(i)$ for $Q_k(i)$ and $X_i(n)$ for $X_i$ to indicate dependence on $n$.
\begin{lemma}
\label{lem:eigenvalues_recurrence}
For $1\le i,k\le n$, we have
\[
Q^{(n)}_k(i)=Q^{(n)}_k(i-1)+(-q)^{2n-i}\,Q^{(n-1)}_{k-1}(i-1).
\]
\end{lemma}
\begin{proof}
We have
\[
Q_k(i)=\sum_{A\in X_k(n)}\langle A,S\rangle,
\]
where $S$ is an arbitrary element of $X_i(n)$. Take $S\in X_i(n)$ to be the diagonal matrix with diagonal $(1,\dots,1,0,\dots,0)$ and let $S'\in X_{i-1}(n-1)$ be the diagonal matrix with diagonal $(1,\dots,1,0,\dots,0)$. For an $n\times n$ Hermitian matrix~$A$, we write
\begin{equation}
A=\begin{pmatrix}
a   & v^* \\
v   & B
\end{pmatrix},   \label{eqn:matrix_A}
\end{equation}
so that $a\in\F_q$, $v\in(\F_{q^2})^{n-1}$, and $B$ is Hermitian of size $(n-1)\times (n-1)$. Then
\begin{align}
Q^{(n)}_k(i-1)-Q^{(n)}_k(i)&=\sum_{A\in X_k(n)}(\langle B,S'\rangle-\langle A,S\rangle)   \nonumber\\
&=\sum_{A\in X_k(n)}(\langle B,S'\rangle(1-\chi(a)).   \label{eqn:Qki_diff}
\end{align}
If $a=0$, then the summand is zero. Thus, to evaluate the sum, we may assume that $A$ is such that $a\in\F_q^*$. For $A$ of the form~\eqref{eqn:matrix_A}, write
\[
L=\begin{pmatrix}
1        & 0\\
-a^{-1}v & I
\end{pmatrix}
\]
(where $I$ is the identity matrix of size $(n-1)\times(n-1)$). Then $L$ is nonsingular and
\[
LAL^*=\begin{pmatrix}
a & 0\\
0 & C
\end{pmatrix},\quad\text{where}\quad C=B-a^{-1}vv^*.
\]
If $A$ has rank $k$, then $C$ has rank $k-1$ since $a$ is nonzero. Hence we find from~\eqref{eqn:Qki_diff} that
\begin{align}
Q^{(n)}_k(i-1)-Q^{(n)}_k(i)&=\!\!\!\!\sum_{C\in X_{k-1}(n-1)}\sum_{v\in(\F_{q^2})^{n-1}}\sum_{a\in\F_q^*}\langle C+a^{-1}vv^*,S'\rangle(1-\chi(a))   \nonumber\\
&=\!\!\!\!\sum_{C\in X_{k-1}(n-1)}\!\!\!\!\!\langle C,S'\rangle\sum_{a\in\F_q^*}(1-\chi(a))\sum_{v\in(\F_{q^2})^{n-1}}\langle a^{-1}vv^*,S'\rangle,   \label{eqn:rec_Q_sums}
\end{align}
using the homomorphism property~\eqref{eqn:ip_homomorphism}. We have
\[
\sum_{C\in X_{k-1}(n-1)}\langle C,S'\rangle=Q^{(n-1)}_{k-1}(i-1)
\]
and
\begin{align*}
\sum_{v\in(\F_{q^2})^{n-1}}\langle a^{-1}vv^*,S'\rangle&=q^{2n-2i}\sum_{v_1,\dots,v_{i-1}\in\F_{q^2}}\chi(a^{-1}(v_1\overline{v_1}+\cdots+v_{i-1}\overline{v_{i-1}}))\\[1ex]
&=q^{2n-2i}(-q)^{i-1}
\end{align*}
since, for an arbitrary nontrivial character $\psi$ of $(\F_q,+)$, we have 
\[
\sum_{v\in\F_{q^2}}\psi(v\overline{v})=1+(q+1)\sum_{v\in \F_q^*}\psi(v)=-q.
\]
Moreover
\[
\sum_{a\in\F_q^*}(1-\chi(a))=q.
\]
Substitute everything into~\eqref{eqn:rec_Q_sums} to find that
\[
Q^{(n)}_k(i-1)-Q^{(n)}_k(i)=q^{2n-2i+1}(-q)^{i-1}Q^{(n-1)}_{k-1}(i-1),
\]
as required.
\end{proof}
\par
To obtain the explicit expression~\eqref{eqn:explicit_ev} for the eigenvalues, we use the recurrence of Lemma~\ref{lem:eigenvalues_recurrence} together with the initial values
\begin{align}
Q_0(i)&=1,   \nonumber\\[0.5ex]
Q_k(0)&=\abs{X_k},   \label{eqn:Qk0}
\end{align}
which follow directly from~\eqref{eqn:eigenvalues}. It is well known (see, for example,~\cite{CarHod1955} or~\cite{Sta1981} for odd $q$ and \cite{Mer2004} for the general case) that 
\begin{equation}
\abs{X_k}=(-1)^k{n\brack k}\prod_{j=0}^{k-1}((-q)^n+(-q)^j).   \label{eqn:num_her_mat}
\end{equation}
We first verify that~\eqref{eqn:explicit_ev} gives the correct expressions for $Q_0(i)$ and $Q_k(0)$. The expression for $Q_0(i)$ holds trivially. Apply the following version of the $q$-binomial theorem
\[
\sum_{j=0}^h(-q)^{h-j\choose 2}{h\brack j}x^j\,y^{h-j}=\prod_{j=0}^{h-1}(x+(-q)^jy)\quad\text{for real $x,y$},
\]
to~\eqref{eqn:num_her_mat} to find that
\[
\abs{X_k}=(-1)^k{n\brack k}\sum_{j=0}^k(-q)^{k-j\choose 2}{k\brack j}(-q)^{nj}.
\]
Using the identity
\[
{n\brack k}{k\brack j}={n-j\brack n-k}{n\brack j},
\]
we see that~\eqref{eqn:explicit_ev} gives the correct expression for $Q_k(0)$. Now invoke Lemma~\ref{lem:eigenvalues_recurrence} and the following version of Pascal's triangle identity
\[
{n-i+1\brack j}-(-q)^{n-i-j+1}{n-i\brack j-1}={n-i\brack j}
\]
to conclude that~\eqref{eqn:explicit_ev} gives the correct expression for $Q_k(i)$ for all $k,i\ge 0$.



\providecommand{\bysame}{\leavevmode\hbox to3em{\hrulefill}\thinspace}
\providecommand{\MR}{\relax\ifhmode\unskip\space\fi MR }
\providecommand{\MRhref}[2]{%
  \href{http://www.ams.org/mathscinet-getitem?mr=#1}{#2}
}
\providecommand{\href}[2]{#2}

\end{document}